\documentclass{article}

\usepackage[utf8]{inputenc}
\usepackage[english]{babel}
\usepackage{amsfonts}
\usepackage[11pt]{moresize}
\usepackage{pgfplots}
\usepackage{xfrac}
\newcommand{\Addresses}{{
		\bigskip
		\footnotesize
		
		Humboldt-Universit\"at zu Berlin, Institut f\"ur Mathematik, Rudower Chausee 25
			\hfill \newline\texttt{}
			\indent 12489 Berlin, Germany} 
		\par\nopagebreak
		\textit{E-mail address}: \texttt{andreibud95@protonmail.com}
	}

\usepackage{amssymb}
\usepackage{mathtools}
\usepackage{tikz}
\usepackage{tikz-cd}
\usepackage{mathtools}
\usepackage{amsmath}
\usepackage{amsthm}
\usepackage{hyperref}
\usepackage{float}
\usetikzlibrary{calc, positioning,decorations.markings,arrows}

\DeclarePairedDelimiter\ceil{\lceil}{\rceil}

\usepackage{graphics}
\usepackage{color}
\usepackage{tabularx,colortbl}
\usetikzlibrary{positioning}
\usetikzlibrary{decorations.markings,arrows}
\usetikzlibrary{calc}
\usepackage{amsthm}

\theoremstyle{plain}
\newtheorem{trm}{Theorem}[section]

\newtheorem{prop}[trm]{Proposition}

\theoremstyle{definition}
\newtheorem{defi}[trm]{Definition}
\newtheorem{rmk}[trm]{Remark}
\def\cpp{\overline{\mathcal{C}^n\mathcal{R}}}
\def\cplow{\overline{\mathcal{C}^{n-s+1}\mathcal{R}}}
\def\pp{\mathcal{C}^n\mathcal{R}}
\def\gp{\overline{\mathcal{GP}}}
\def\pone{{\mathbb P}^1}

\def\OO{\mathcal{O}}

\def\cM{\mathcal{M}}
\def\cR{\mathcal{R}}
\def\rr{\overline{\mathcal{R}}}

\def\cU{\mathcal{U}}
\def\cC{\mathcal{C}}

\def\Pic0{{\rm Pic}^0(X)}

\def\mm{\overline{\mathcal{M}}}

\def\zz{\overline{\mathcal{Z}}}

\begin{document}
\title{The birational geometry of $\rr_{g,2}$ and Prym-canonical divisorial strata}
\author{Andrei Bud}
\date{}
\maketitle
 \begin{abstract}
 We prove that the moduli space of double covers ramified at two points $\cR_{g,2}$ is uniruled for $3\leq g\leq 6$ and of general type for $g\geq 16$. Furthermore, we consider Prym-canonical divisorial strata in the moduli space $\cpp_g$ parametrizing $n$-pointed Prym curves, and we compute their classes in $\mathrm{Pic}_\mathbb{Q}(\cpp_g)$.  
 \end{abstract}

\section{Introduction}
In his fundamental paper \cite{MumfordPrym}, Mumford initiated the study of double covers as a way of understanding polarized Abelian varieties. It is then natural to consider the moduli space $\mathcal{R}_{g,2n}$ parametrizing double covers ramified at $2n$ points, and describe its birational geometry. The classical case $\mathcal{R}_g$ when the cover is unramified has received considerable attention. When the genus $g$ is small, it is known that $\mathcal{R}_g$ is rational for $g =2,3,4$ (cf. \cite{Dolgachev}, the references therein and \cite{Catanese}), unirational for $g = 5, 6, 7$ (cf. \cite{Izadi}, \cite{VerraA4}, \cite{DonagiA5}, , \cite{MoriMukai}, \cite{VerraA5} and \cite{FarVerNikulin}) and uniruled for $g =8$ (cf. \cite{FarVerNikulin}). The situation changes for higher genus and we know that $\mathcal{R}_g$ is of general type when $g\geq 13, g\neq 16$ (cf. \cite{Bruns} and \cite{FarLud}). Apart from one exotic case in genus $2$, see \cite{Ortega-triplecov}, the only other way to obtain principally polarized Abelian varieties is by considering double covers ramified at two points. 

By the theory of double covers, the moduli space $\cR_{g,2}$ can be alternatively described as 
\[ \cR_{g,2} \coloneqq \left\{ [C,x+y, \eta] \ | \ [C]\in \cM_g, x,y\in C \ \mathrm{and} \ \eta\in \mathrm{Pic}^{-1}(C) \ \mathrm{satisfying} \ \eta^{\otimes 2}\cong \OO_C(-x-y) \right\} \]
where $x$ and $y$ correspond to the two branch points of the associated cover and their order is irrelevant. We will call such a triple $[C, x+y,\eta]$ a $2$-branched Prym curve.

One important feature of the moduli space $\cR_{g,r}$ is that it comes with the Prym map 
\[ \mathcal{P}_{g,r}\colon \cR_{g,r} \rightarrow \mathcal{A}^{\delta}_{g-1+\frac{r}{2}}\]
to the moduli space of Abelian varieties of dimension $g-1+\frac{r}{2}$ equipped with a polarization of type $(1,\ldots,1,2,\ldots,2)$ where $2$ appears $g$ times. This map received considerable attention in recent years, see \cite{Marcucci-Pirola}, \cite{Ortega-Naranjo-Verra}, \cite{Prym-Torelli}; adding to the vast literature on the Prym map in the unbranched case, see \cite{Beau77}, \cite{Donagi-Smith} and \cite{Don92} among many others. 

Our interest in the case $r=2$ is motivated by the fact that $r=0$ and $r=2$ are the only two cases when $\mathcal{P}_{g,r}$ provides a correspondence between double covers and principally polarized Abelian varieties, as first pointed out in \cite{MumfordPrym}. Our main result is the following: 

\begin{trm} \label{kodairarr}The moduli space $\cR_{g,2}$ is of general type for $g\geq 16$ and $\cR_{13,2}$ has non-negative Kodaira dimension.  
\end{trm} 

There are three main ideas of the proof. First, we consider a suitable compactification $\rr_{g,2}$ of $\cR_{g,2}$, following the method outlined in \cite{corn} and \cite{Casa}. Secondly, we show that the canonical class $K_{\rr_{g,2}}$ is big and lastly, we show that the singularities of $\rr_{g,2}$ are mild enough in order to extend holomorphically the pluricanonical forms of $\rr^{\mathrm{reg}}_{g,2}$ to any desingularisation. For this last step, we follow closely \cite{Ludspin} and \cite{FarLud}. 

To show that $K_{\rr_{g,2}}$ is big, we will use pullbacks of divisors through the map $\rr_{g,2} \rightarrow \mm_{2g}$ retaining the source of the double cover. The image of the map is not contained in any Brill-Noether or Gieseker-Petri divisor, see Theorem \ref{transversal-bn-gp}. This is in sharp contrast with the situation in the unramified case, where the Brill-Noether properties of a generic double cover depend on the parity of the genus of the base, see \cite[Theorem 0.4]{FarAprGreencong}.

Next, we are interested in the birational geometry of $\cR_{g,2}$ when the genus $g$ is small. We have that: 
\begin{trm} \label{lowgenus}
	The moduli space $\cR_{g,2}$ is uniruled for $3\leq g\leq 6$.  
\end{trm}

This result is obtained by relating the moduli space $\cR_{g,2}$ to strata parametrizing divisors of quadratic differentials, which we know from \cite{BAR18} to be uniruled when $3\leq g\leq 6$. 

In the second part of this paper, we investigate further the relation between Prym curves and quadratic differentials. To set things up, we introduce the moduli space $\mathcal{C}^n\mathcal{R}_g$ parametrizing tuples $[X, x_1,\ldots, x_n, \eta]$ where $[X,\eta]$ is an element of the Prym variety $\cR_g$ and $x_1,\ldots,x_n$ are distinct points on $X$. For a positive partition $\underline{d} = (d_1,\ldots, d_n)$ of $g-1$, we consider the divisor $PD_{\underline{d}}$ in $\mathcal{C}^n\mathcal{R}_g$ defined as: 
\[ PD_{\underline{d}} \coloneqq \left\{ [X,x_1,\ldots, x_n,\eta] \in \mathcal{C}^n\mathcal{R}_g \ | \ h^0\text{\large(}X, \omega_X\otimes\eta(-\sum_{i=1}^nd_ix_i)\text{\large)} \geq 1 \right\}\] 

We consider a suitable compactification $\overline{\mathcal{C}^n\cR}_g$ of $\mathcal{C}^n\cR_g$ and compute the class of the divisor $\overline{PD}_{\underline{d}}$ in this space. We obtain: 
\begin{trm} \label{prymquad}
	Let $\underline{d} = (d_1,\ldots,d_n)$ a partition of $g-1$ with all entries positive. The class of the Prym-canonical divisorial stratum $\overline{PD}_{\underline{d}}$ in $\mathrm{Pic}_{\mathbb{Q}}(\overline{\mathcal{C}^n\mathcal{R}}_g)$ is given by: 
	\[[\overline{PD}_{\underline{d}}] = -\lambda + \sum_{i=1}^n\frac{d_i(d_i+1)}{2}\psi_i +\frac{1}{4}\delta_0^{\mathrm{ram}}-\sum_{\substack{1\leq i \leq g-1 \\ d_S\geq i-1 }} \binom{d_S-i+2}{2}\delta_{i,S} - \sum_{\substack{1\leq i \leq g \\  d_S \leq i-1 }} \binom{i-d_S}{2}(\delta_{i,S:g-i} + \delta_{i,S})  \]
	where $d_S \coloneqq \sum_{i\in S} d_i$ and $\delta_{0,S:g} \coloneqq 0$.
\end{trm}

Note that the coefficients of $\delta_0'$ and $\delta_0''$ are $0$. For the definition of the classes appearing in Theorem \ref{prymquad} we refer to Section 6. These divisors can be seen as a Prym analogue of the canonical divisorial strata appearing in \cite{MulMark}, \cite{Logan}, \cite{FabianMuller} and \cite{GrushevskyZakharov}. Moreover, these divisors are closely related to the divisorial strata of quadratic differentials, see \cite[Proposition 1.4]{Mullanek-diff}. The study of such divisors led to important results in understanding the geometric aspects of the moduli space $\overline{\mathcal{M}}_{g,n}$ such as the Kodaira dimension (cf. \cite{Logan}) and the effective cone (cf. \cite{Mullanek-diff}).

To prove Theorem \ref{prymquad}, we consider suitable maps  $\pi_1\colon \mm_{g-i,n+1-s} \rightarrow \cpp_g $ and $\pi_2\colon\cplow_{g-i} \rightarrow \cpp_g$. Understanding the pullbacks at the level of rational Picard groups is enough to compute all coefficients of $[\overline{PD}_{\underline{d}}]$ but the one of $\delta_0^{\mathrm{ram}}$. Lastly, we use \cite[Proposition 1.4]{Mullanek-diff} to conclude the theorem.

As torsion classes are irrelevant to us, the Picard groups will be considered over $\mathbb{Q}$ throughout the paper.

\textbf{Acknowledgements:} I would like to thank my advisor Gavril Farkas for choosing this interesting topic and for all his insightful contributions. I am grateful to Scott Mullane, whose comments led to significant improvements in this paper. I have also benefited from discussions with Carlos Maestro P\'erez and Johannes Schmitt on topics related to this article.

\section{A compactification of the moduli space $\rr_{g,2}$} 

We are interested in compactifying the moduli space $\cR_{g,2}$ parametrizing smooth $2$-branched Prym curves. The way we do this is similar to the approaches in \cite{corn} and \cite{Casa} and it inspires us to consider the following definitions:
\begin{defi}
	Let $[X, x+y]$ be a pointed semistable Deligne-Mumford curve (with the two points unordered) and let $E$ be an irreducible component of $X$. We say that $E$ is exceptional if $E$ is smooth, rational, the points $x, y$ are not on $E$ and $|E \cap \overline{X\setminus E}| = 2$. We say that $[X, x+y]$ is quasistable if any two distinct exceptional components do not intersect. 
\end{defi}

We are now ready to extend the definition of a $2$-branched Prym curve to singular curves. 
\begin{defi}
	We define a $2$-branched Prym curve of genus $g$ to be the data $[X, x+y, \eta, \beta]$, where $[X, x+y]$ is a genus $g$ quasistable curve, $\eta \in \textrm{Pic}(X)$ and $\beta\colon \eta^{\otimes 2} \rightarrow \OO_X(-x-y)$ is a morphism of invertible sheaves satisfying: 
	\begin{enumerate}
\item The sheaf $\eta$ has total degree $-1$ and has degree $1$ on each exceptional component, 
\item The morphism $\beta$ is non-zero at a general point of a non-exceptional component of $X$. 
\end{enumerate}
\end{defi} 

 In the above setting, consider $E_1,\ldots, E_n$ the exceptional components of $[X, x+y]$ and let $\tilde{X} \coloneqq \overline{X\setminus \cup_{i=1}^n E_i}$. We denote by $q_i^1$ and $q_i^2$ the intersection of $E_i$ with $\tilde{X}$ and we get an isomorphism 
\[ \beta_{\tilde{X}}\colon \eta^{\otimes2}_{|\tilde{X}} \rightarrow \OO_{\tilde{X}}\text{\large(}-x-y-\sum_{i=1}^{n}(q_i^1 +q_i^2) \text{\large)} \]  

In particular, when $X$ is smooth, we obtain that $\eta$ is a root of order 2 of $\OO_X(-x-y)$.  
Next, we define the notion of isomorphism between two $2$-branched Prym curves. 

\begin{defi} \label{definitionauto}
	We say that two $2$-branched Prym curves  $[X, x+y, \eta, \beta]$ and  $[X', x'+y', \eta', \beta']$ are isomorphic if there exists an isomorphism $\sigma\colon X \rightarrow X'$ such that 
	\begin{enumerate}
		\item it sends $x+y$ to $x'+y'$ 
		\item there exists an isomorphism $\tau\colon \sigma^*\eta' \rightarrow \eta$ making the following diagram commutative 	
		\[
		\begin{tikzcd}
		(\sigma^*\eta')^{\otimes2} \arrow{r}{\tau^{\otimes2}}  \arrow[swap]{d}{\sigma^*(\beta')} & \eta^{\otimes2} \arrow{d}{\beta} \\
		\sigma^*\text{\large(}\OO_{X'}(-x'-y')\text{\large)}\arrow{r}{\sim}& \OO_X(-x-y)
		\end{tikzcd}
		\] 
	\end{enumerate}
	Moreover, we say that an automorphism is inessential if it induces the identity on the stable model of $[C, x+y]$. 
\end{defi}
 
The results of \cite{corn} can be easily adapted to our situation and we obtain a compactification $\rr_{g,2}$ of $\cR_{g,2}$, parametrizing isomorphism classes of $2$-branched Prym curves. As in \cite{corn}, we obtain that the space $\rr_{g,2}$ is normal and projective. Moreover, it is irreducible as it is birational to the irreducible divisor $\Delta_0^{\mathrm{ram}}$ in $\rr_{g+1}$ (see \cite[page 9]{MiraBern} for irreducibility).

 We consider the map forgetting the $2$-branched Prym structure
\[\pi_{g,2}\colon\rr_{g,2} \rightarrow \mm_{g,2/\mathbb{Z}_2}  \]
and we will describe the boundary divisors of $\rr_{g,2}$ lying above each boundary component of $\mm_{g,2/\mathbb{Z}_2}$.

1. Consider a generic element $[X/t_1\sim t_2, x+y]$ of the divisor $\Delta_0$ in $\mm_{g,2/\mathbb{Z}_2}$. Over $\Delta_0$ we have two divisors: $\Delta_0'$ and $\Delta_0^{\mathrm{ram}}$.

$\bullet$ The divisor $\Delta_0'$ contains the pairs $[X/t_1\sim t_2,x+y, \eta]$ satisfying that for the normalization map $\nu\colon X \rightarrow X/t_1\sim t_2 $ we have $(\nu^{*}\eta)^{\otimes2} \cong \mathcal{O}_X(-x-y)$. 

$\bullet$ The divisor $\Delta_0^{\mathrm{ram}}$ contains the pairs $[X\cup R/{t_1\sim r_1, t_2\sim r_2}, x+y, \eta]$, where $R$ is an exceptional component, $\eta_{|R} \cong \OO_R(1)$ and $\eta_{|X}^{\otimes 2} \cong \OO_X(-x-y-t_1-t_2)$. 

It is immediate to see that  $\deg(\Delta_0'/\Delta_0) = 2^{2g-1}$ and $\deg(\Delta_0^{\mathrm{ram}}/\Delta_0) = 2^{2g-2}$. Furthermore $\Delta_0^{\mathrm{ram}}$ is the ramification divisor of $\pi_{g,2}$ and has ramification order 2.

2. Consider  $[X\cup_{x'\sim y'} Y,  x+y]$ a generic element of $\Delta_{i,\left\{1\right\}}$, where $g(X) = i$ and $g(Y) = g-i$ and assume $x, x'\in X$. Then there is a unique divisor in $\rr_{g,2}$ lying above $\Delta_{i,\left\{1\right\}}$, which we will denote $\Delta_{i:g-i}$. This divisor parametrizes pairs $[X\cup_{x'\sim r_1} R\cup_{r_2\sim y'} Y,  x+y, \eta ]$ satisfying that $R$ is an exceptional component, $\eta_{|R} \cong \OO_R(1)$, $\eta_{|X}^{\otimes 2} \cong \OO_X(-x-x')$ and $\eta_{|Y}^{\otimes 2} \cong \OO_X(-y-y')$.

3. Consider $[X\cup_{x'\sim y'}Y, x_1+x_2]$ a generic element of the boundary divisor $\Delta_{i,\left\{1,2\right\}}$, where $g(X) = i$, $g(Y) = g-i$ and $x', x_1, x_2 \in X$. Then there are two divisors $\Delta_{i:g-i,\left\{\OO\right\}}$ and $\Delta_{i:g-i,\left\{\eta\right\}}$ lying in $\rr_{g,2}$ above  $\Delta_{i,\left\{1,2\right\}}$. 

$\bullet$ The divisor $\Delta_{i:g-i,\left\{\OO\right\}}$ contains the pairs $[X\cup_{x'\sim y'}Y, x_1+x_2, \eta]$ satisfying that $\eta^{\otimes2}_{|X} \cong \OO_X(-x_1-x_2)$ and $\eta_{|Y} \cong \OO_Y$.

$\bullet$ The divisor $\Delta_{i:g-i,\left\{\eta\right\}}$ contains the pairs $[X\cup_{x'\sim y'}Y, x_1+x_2, \eta]$ satisfying that $\eta^{\otimes2}_{|X} \cong \OO_X(-x_1-x_2)$ and $\eta_{|Y} \in \textrm{Pic}(Y)[2]\setminus \left\{ \OO_Y\right\}$. 

\begin{rmk} \label{irreducible divisors}
	All the boundary divisors of $\rr_{g,2}$ described above are irreducible.
\end{rmk}

The remark follows immediately for almost all boundary divisors by simply noting that $\cM_g$, $\cR_g$ and $\cR_{g,2}$ are irreducible. The only divisor for which Remark \ref{irreducible divisors} requires more attention is $\Delta^{\mathrm{ram}}_0$. That $\Delta^{\mathrm{ram}}_0$ is irreducible is deduced from the following proposition. 
\begin{prop}
	The moduli space $\cR_{g,2n}$ is irreducible for all $g \geq 2, n\geq 0$. 
\end{prop}
\begin{proof} We will prove the proposition using an inductive argument. The cases $n = 0$ and $n = 1$ are already covered, hence we can assume that $n\geq 2$. For a given $g$, consider the smallest $n$ for which $\cR_{g,2n}$ is not irreducible.
	
We consider the moduli space $\cR_{g,2n}'$ parametrizing pairs $[C, x_1,\ldots, x_{2n}, \eta]$ where $[C] \in \cM_g$, the points $x_1,\ldots, x_{2n} \in C$ are pairwise distinct and $\eta \in \textrm{Pic}^{-n}(C)$ such that $\eta^{\otimes2} \cong \OO_C(-x_1-\cdots - x_{2n})$.  Because the approach in \cite{corn} applies with little change to this case, we obtain a compactification $\rr_{g,2n}'$ and, in particular a map 
	\[ \pi\colon \cR_{g,2n-2}' \times \cM_{0,4} \rightarrow \rr_{g,2n}'  \]
	given as 
	\[ ([C, x_1, \ldots, x_{2n-3}, t_1, \eta_C], [\mathbb{P}^1, t_2, x_{2n-2}, x_{2n-1}, x_{2n}]) \mapsto [C\cup_{t_1\sim r_1} R \cup_{r_2\sim t_2} \mathbb{P}^1, x_1,\ldots, x_{2n}, \eta ]  \]
	where $R$ is an exceptional component and the line bundle $\eta$ is defined by 
	\[ \eta_{|C} \cong \eta_C, \ \  \eta_{|\mathbb{P}^1} \cong \OO_{\mathbb{P}^1}(-2) \ \ \textrm{and} \ \eta_{|R} \cong \OO_R(1) \] 
	
	We know from the approach in \cite{corn} that $\rr_{g,2n}'$ is given locally as the quotient of the base of a universal deformation by the automorphism group of the $2n$-branched Prym curve (where the branch points are ordered). Because a generic element in $\cM_{0,4}$ and in $\cR_{g,2n-2}'$ has no non-trivial automorphisms it follows that a generic element in $\textrm{Im}(\pi)$ has no inessential automorphisms. 
	
	In particular, a generic element in $\textrm{Im}(\pi)$ is smooth. If we consider the finite map of degree $2^{2g}$ 
	\[  \rr_{g,2n}' \rightarrow \mm_{g,2n} \]
	obtained by forgetting the $2n$-Prym structure, we observe that $\textrm{Im}(\pi)$ has degree $2^{2g}$ over the divisor $\Delta_{0, \left\{2n-2, 2n-1, 2n\right\}}$ in $\mm_{g,2n}$. 
	
	Because $\textrm{Im}(\pi)$ is irreducible (from the induction hypothesis), contains a smooth point of $\rr'_{g,2n}$ and has degree $2^{2g}$ over its image in $\mm_{g,2n}$, it follows immediately that $\rr'_{g,2n}$ is irreducible (otherwise $\textrm{Im}(\pi)$ would be in the intersection of all irreducible components and hence it would be impossible to contain smooth points). 
	
	Because we have an obvious surjective map $\cR'_{g,2n} \rightarrow \cR_{g,2n}$, the conclusion follows. 
\end{proof}
\section{Maps between moduli spaces} As easily remarked, there is an obvious map $i\colon\rr_{g,2} \rightarrow \Delta_0^{\mathrm{ram}}\subseteq \rr_{g+1}$ obtained by glueing an exceptional component to the two marked points. This map fits into a commutative diagram 
	\[
\begin{tikzcd}
\rr_{g,2} \arrow{r}{i}  \arrow[swap]{d}{\pi_{g,2}} & \rr_{g+1} \arrow{d}{\pi_{g+1}} \\
\mm_{g,2/\mathbb{Z}_2}\arrow{r}{i'}& \mm_{g+1}
\end{tikzcd}
\] 
We are interested in describing the pullback map $i^{*}\colon \textrm{Pic}(\rr_{g+1}) \rightarrow \textrm{Pic}(\rr_{g,2})$. For this, we first set some notations.

The Picard group of $\mm_{g,2/\mathbb{Z}_2}$ injects in $\textrm{Pic}(\mm_{g,2})$ as the subgroup of $\mathbb{Z}_2$-invariant classes. Hence, Pic$(\mm_{g,2/\mathbb{Z}_2})$ is generated by $\psi = \psi_1 + \psi_2$, the class $\lambda$ and the boundary divisors (for which we preserve the notation from $\mm_{g,2}$). We denote again by $\psi$ and $\lambda$, the pullbacks by $\pi_{g,2}$ of the respective classes.

\begin{rmk} \label{description of i}
Because the pullback maps $i'^{*}, \pi_{g,2}^{*}$ and $\pi_{g+1}^{*}$ at the level of Picard groups are explicitly known and because $i^{*}\circ \pi_{g+1}^{*} = \pi_{g,2}^{*}\circ i'^{*}$, we conclude that: 
\[ i^{*}\lambda = \lambda, \ \ i^{*}\delta_0^{\mathrm{ram}} = -\frac{1}{2}\psi +\delta_0^{\mathrm{ram}} + \sum \delta_{i:g-i}, \ \ i^{*}\delta_0'' = 0, \ \ i^{*}\delta_0' = \delta_0' \]
\[ i^{*}\delta_i = \delta_{i-1:g-i+1, \left\{\OO\right\}} \ \ \textrm{and} \ i^{*}\delta_{i:g+1-i} = \delta_{i-1:g-i+1, \left\{\eta\right\}} + \delta_{g-i:i, \left\{\eta\right\}} \]
\end{rmk}

 We remark that the computation above is done at the level of moduli stacks (not coarse). In this situation we have $\pi_{g,2}^{*}\delta_{i,\left\{1\right\}} = 2\delta_{i:g-i} = [\Delta_{i:g-i}]$.

Having an element $[C, x+y, \eta,\beta] \in \rr_{g,2}$, we obtain a degree 2 map $\pi\colon \tilde{C} \rightarrow C$ that is ramified only above $x, y$ and eventually above the nodes of $C$. Seeing the space $\rr_{g,2}$ as parametrizing such admissible covers $\pi\colon \tilde{C} \rightarrow C$, we get a map 
\[  \mathcal{X}\colon\rr_{g,2} \rightarrow \mm_{2g,2/\mathbb{Z}_2} \]
sending $[\pi\colon \tilde{C} \rightarrow C]$ to $[\tilde{C}, \tilde{x}+ \tilde{y}]$ where $\tilde{x}$ and $\tilde{y}$ are the two smooth ramification points of $\tilde{C}$. 
Forgetting the points we obtain a map 
\[\mathcal{X}_{g,2}\colon \rr_{g,2} \rightarrow \mm_{2g} \]

This is a pointed version of the map  $\mathcal{X}_{g+1}$ considered in \cite{MiraBern} and \cite{FarLud}. In fact, we have the obvious commutative diagram 
	\[
\begin{tikzcd}
\rr_{g,2} \arrow{r}{}  \arrow[swap]{d}{i} & \mm_{2g,2/\mathbb{Z}_2} \arrow{d}{i'} \arrow{r}{} & \mm_{2g} \\
\rr_{g+1}\arrow{r}{\mathcal{X}_{g+1}}& \mm_{2g+1}
\end{tikzcd}
\] 
Our next task is to describe the map $\mathcal{X}^*_{g,2}\colon \textrm{Pic}(\mm_{2g}) \rightarrow \textrm{Pic}(\rr_{g,2})$. Because we know the maps $i^{*}, i'^{*}$ and $\mathcal{X}_{g+1}^{*}$ at the level of Picard groups, we can immediately see that 
\[\mathcal{X}_{g,2}^{*}\lambda = 2\lambda - \frac{1}{4} \delta_0^{\mathrm{ram}} - \frac{1}{4} \sum \delta_{i:g-i} + \frac{1}{8}\psi \] 

To compute the pullback of the boundary divisors, it suffices to apply the same method as in \cite{MiraBern} and reduce the problem to a simple count of the number of nodes. To exemplify this, let $B$ be a disk transverse to a general point of $\Delta_0^{\mathrm{ram}}$. The map $\mathcal{X}_{g,2}$ sends this general point to a point of $\Delta_0$ in $\mm_{2g}$ having a unique node. It follows that $\mathcal{X}_{g,2*} B \cdot \delta_0 = 1$. Hence the coefficient of $\delta^{\mathrm{ram}}_0$ in $\mathcal{X}_{g,2}^{*}\delta_0$ is $1$. 
Proceeding as in this example we obtain: 
\begin{align}\mathcal{X}_{g,2}^{*}\delta_0 & = \delta_0^{\mathrm{ram}} + 2\delta_0' + 2\sum \delta_{i:g-i, \left\{\eta\right\}} \nonumber \\
\mathcal{X}_{g,2}^{*}\delta_i & = 2\delta_{g-i:i, \left\{\OO\right\}} \ \ \textrm{if} \ i \ \textrm{is odd} \nonumber 
\end{align}
and
\[ \mathcal{X}_{g,2}^{*}\delta_i  = 2\delta_{g-i:i, \left\{\OO\right\}} + \delta_{\frac{i}{2}:g-\frac{i}{2}} \ \ \textrm{if} \ i \ \textrm{is even} \]

In fact, we can compute the pullback $\mathcal{X}^{*}\colon \textrm{Pic}(\mm_{2g,2/\mathbb{Z}_2}) \rightarrow \textrm{Pic}(\rr_{g,2})$ and obtain for $i\leq g$: 
\begin{align*} 
&\mathcal{X}^{*}\delta_{i,\left\{1\right\}} = \delta_{\frac{i}{2}:g-\frac{i}{2}} \ \ \textrm{if} \ i \ \textrm{is even} \\  
&\mathcal{X}^{*}\delta_{i,\left\{1\right\}} = 0 \ \ \textrm{if} \ i \ \textrm{is odd} \\
 &\mathcal{X}^{*}\delta_{i,\emptyset} = 2\delta_{g-i:i,\left\{\OO\right\}} \\
 &\mathcal{X}^{*}\delta_{i,\left\{1,2\right\}} = 0	
\end{align*} 
We can easily check that the commutativity $\mathcal{X}^{*}\circ i'^{*} = i^{*}\circ \mathcal{X}_{g+1}^{*}$ is respected. 

The formula for $\mathcal{X}_{g,2}^{*}\lambda$ can be alternatively computed by considering a family in $\rr_{g,2}$ given as 

\[
\begin{tikzcd}
\tilde{\mathcal{C}}  \arrow{r}{\pi} \arrow[swap]{dr}{\tilde{f}} & \mathcal{C} \arrow{d}{f} \\
& \Delta_{t_1}
\end{tikzcd}
\]
We consider $D_2\subseteq \tilde{\mathcal{C}}$ the locus where $\pi$ has order $2$ and let $C_2$ be its image in $\cC$. 

We have 
\[ f_{*}\circ \pi_{*} \text{\large(}c_1(\omega_{\tilde{f}})\cdot c_1(\omega_{\tilde{f}})\text{\large)} = f_*\circ\pi_*\text{\large(}[\pi^*c_1(\omega_f)+D_2 ]\cdot[\pi^*c_1(\omega_f)+D_2 ]\text{\large)}  \]
Using that $\pi^{*}(C_2) = 2D_2$ and the push-pull formula, this is furthermore equal to 
\[ 2f_*\text{\large(}c_1(\omega_f)\cdot c_1(\omega_f)\text{\large)} + 2f_*\text{\large(}C_2\cdot c_1(\omega_f)\text{\large)} + \frac{1}{2}f_*(C_2\cdot C_2)  \]

It follows that $\mathcal{X}_{g,2}^{*}(\kappa_1)_{\mm_{2g}} = 2(\kappa_1)_{\rr_{g,2}} + \frac{3}{2}\psi$. Using that $(\kappa_1)_{\mm_{2g}} = 12\lambda - \delta$ and  $(\kappa_1)_{\rr_{g,2}} = 12\lambda - \pi_{g,2}^*\delta$ we recover our formula.

We use the notation $\delta$ for the sum of the boundary divisors of the respective moduli space. Next, we compute the canonical class of the variety $\rr_{g,2}$. 
\begin{prop}
	The canonical class $K_{\rr_{g,2}}$ is equal to 
	\[ K_{\rr_{g,2}} = \psi +13\lambda - 2\delta - \delta_0^{\mathrm{ram}} - \delta_{0:g,\left\{\OO\right\}} - \delta_{0:g, \left\{\eta\right\}}- 2\sum \delta_{i:g-i} - \delta_{g-1:1, \left\{\eta\right\}} - \delta_{g-1:1, \left\{\OO\right\}}   \]
\end{prop} 
\begin{proof}
	We consider the map $\pi \colon \mm_{g,2} \rightarrow \mm_{g,2/\mathbb{Z}_2}$ and observe that
	\[ \pi^{*}K_{\mm_{g,2/\mathbb{Z}_2}} = \psi +13\lambda - 2\delta - \delta_{0,\left\{1,2\right\}} - \delta_{1,\emptyset}  \]
	Consequently we get: 
	\[K_{\mm_{g,2/\mathbb{Z}_2}} = \psi +13\lambda - 2\delta - \delta_{0,\left\{1,2\right\}}- \delta_{1,\emptyset}\] 
	Since the map $\pi_{g,2}\colon\rr_{g,2} \rightarrow \mm_{g,2/\mathbb{Z}_2}$ is ramified only along the divisor $\Delta_0^{\mathrm{ram}}$ we get that
	\[ K_{\rr_{g,2}} = \psi +13\lambda - 2\delta - \delta_0^{\mathrm{ram}} - \delta_{0:g,\left\{\OO\right\}} - \delta_{0:g, \left\{\eta\right\}}- 2\sum \delta_{i:g-i} - \delta_{g-1:1, \left\{\eta\right\}} - \delta_{g-1:1, \left\{\OO\right\}}   \]
\end{proof}

\section{The geometry of $\rr_{g,2}$}

One way of obtaining effective divisors of small slope on $\rr_{g,2}$ is by pullback from $\mm_{2g}$. Let $[D]$ be the class of an effective divisor $D$ in $\mm_{2g}$. If the image of $\mathcal{X}_{g,2}$ is not contained in $D$, we conclude that $\mathcal{X}_{g,2}^{*}([D])$ is the class of an effective divisor in $\rr_{g,2}$. 

Consequently, we prove next that $\mathcal{X}_{g,2}(\rr_{g,2})$ is not contained in some well-known divisors of small slope in $\mm_{2g}$. 

\begin{trm} \label{transversal-bn-gp}
	The image of $\mathcal{X}_{g,2}\colon \rr_{g,2}\rightarrow \mm_{2g}$ is not contained in any Brill-Noether or Gieseker-Petri divisor. 
\end{trm}
\begin{proof}
	A point $[C\cup_{p\sim x}\pone, y+z]$ in the divisor $\Delta_{0:g,\left\{\OO\right\}}$ is mapped by $\mathcal{X}_{g,2}$ to $[C_1\cup_{p_1\sim p_2} C_2] \in \mm_{2g}$, where $[C_1,p_1]$ and $[C_2, p_2]$ are two copies of $[C,p] \in \mm_{g,1}$. We show that we can choose $[C,p]$ in such a way that $[C_1\cup_{p_1\sim p_2} C_2]$ is not contained in any Brill-Noether divisor.
	
	Consider a curve $[\pone, p, y_1,\ldots,y_g]\in \cM_{0,g+1}$ to which we glue at each $y_i$ a copy of $[E,z]$, a generic elliptic curve. We denote the curve obtained in this way by $[C,p]$ 
	
	Next, we consider the moduli space $\mm_{0, 2g}$ and the map 
	\[ i\colon \mm_{0, 2g} \rightarrow \mm_{2g}\]
	obtained by glueing a copy of $[E,z]$ at each marking, where $[E,z]$ is again a generic elliptic curve. 
	 
	 Next, observe that for the curve $[C,p]$ we have $[C_1\cup_{p_1\sim p_2} C_2] \in i(\mm_{0, 2g})$. Our conclusion follows from Proposition 4.1 in \cite{EisenbudHarrisg>23} which says that $i(\mm_{0, 2g})$ does not meet any Brill-Noether divisor. 
	 
	 The Gieseker-Petri case can be treated analogously. Using Theorem A$'$ in \cite{EisenHarrisGieseker} we see that we can construct a curve in $\mathcal{X}_{g,2}(\rr_{g,2})$ not contained in any Gieseker-Petri divisor. 
	 
	 The case of the Gieseker-Petri divisor $E^1_{g+1}$ admits another proof, which we will now present. We remark that $E^1_{g+1}$ is the closure of the branch locus of a proper and finite map 
	 \[ \pi\colon \overline{\mathcal{G}}^1_{g+1} \rightarrow \mm_{2g}^{\textrm{ct}} \]
	 where $\overline{\mathcal{G}}^1_{g+1}$ is the map parametrizing limit $g^1_{g+1}$'s on curves of compact type. The degree of this map is known to be the Catalan number $C_g = \frac{1}{g+1}\binom{2g}{g}$. 
	 
	 We consider the curve $[C_1\cup_{p_1\sim p_2}C_2]$ obtained by glueing together two copies of a generic $[C,p] \in \cM_{g,1}$. 
	 
	 Our goal is to show that every crude limit linear series in $G^1_{g+1}([C_1\cup_{p_1\sim p_2}C_2])$ is refined and that the number of $g^1_{g+1}$ over this curve is $\frac{1}{g+1}\binom{2g}{g}$. If we show this, it will follow from Corollary 3.5 in \cite{limitlinearbasic} that this curve is not contained in $E^1_{g+1}$. 
	 
	 We consider a crude limit linear series in $G^1_{g+1}([C_1\cup_{p_1\sim p_2}C_2])$. This is simply a collection $\left\{L_1, V_1\right\}$ and $\left\{L_2, V_2\right\}$ of $g^1_{g+1}$ on the two respective components, satisfying the inequalities 
	 \[ a_0 + b_1 \geq g+1 \ \ \textrm{and} \ a_1+b_0\geq g+1\]
	 where $(a_0,a_1)$ and $(b_0, b_1)$ are the two vanishing sequences at the points $p_1$ and $p_2$ respectively. Because $[C,p]\in \cM_{g,1}$ is generic, Theorem 4.5 in \cite{limitlinearbasic} implies that $a_0+a_1 \leq g+1$ and $b_0+b_1 \leq g+1$. 
	 
	 It follows that 
	 \[ a_0+b_1 = a_1+b_0 = a_0+a_1= b_0+b_1 = g+1\]
	 and implicitly all the crude limit linear series in $G^1_{g+1}([C_1\cup_{p_1\sim p_2}C_2])$ are refined. 
	 
	 We proceed to count them: Denoting $d= g+1-a_0$ we observe that $\left\{L_1(-a_0p_1), V_1(-a_0p_1)\right\}$ and $\left\{L_2(-a_0p_2), V_2(-a_0p_2)\right\}$ are $g^1_d$'s having order $2d-g-1$ at $p_1$ and $p_2$ respectively. 
	 
	 Theorem A in \cite{KodMg} implies there are $(2d-g-1)\frac{g!}{d!(g-d+1)!}$ such $g^1_d$'s over a generic curve $[C,p] \in \cM_{g,1}$. Consequently, we get in this way 
	 \[ \sum\limits_{d = \ceil*{\frac{g+1}{2}}}^{g+1} \frac{(2d-g-1)^2}{(g+1)^2} \binom{g+1}{d}^2 \] 
	 
	 distinct limit linear series in $G^1_{g+1}([C_1\cup_{p_1\sim p_2}C_2])$. We can rewrite this sum as 
	 \[ \sum\limits_{d = \ceil*{\frac{g+1}{2}}}^{g+1}\textrm{\LARGE[}\binom{g}{d}-\binom{g}{d-1}\textrm{\LARGE]}^2 = \frac{1}{2}\sum\limits_{d = 0}^{g+1}\textrm{\LARGE[}\binom{g}{d}-\binom{g}{d-1}\textrm{\LARGE]}^2\]
	 
	 Using the Narayana-Catalan identity (cf. A46 in \cite{StanleyCatalan})
	 \[ \sum\limits_{d=1}^g\binom{g}{d}\binom{g}{d-1} = \frac{g}{g+1} \binom{2g}{g} \]
	 and the classical identity 
	 \[ \sum\limits_{d=0}^g \binom{g}{d}^2 = \binom{2g}{g}\]
	 we conclude that the cardinality of $G^1_{g+1}([C_1\cup_{p_1\sim p_2}C_2])$ is 
	 \[ \sum\limits_{d = \ceil*{\frac{g+1}{2}}}^{g+1} \frac{(2d-g-1)^2}{(g+1)^2} \binom{g+1}{d}^2  = \frac{1}{g+1}\binom{2g}{g} \]
	 as required.
\end{proof}

We remark that Theorem \ref{transversal-bn-gp} is in stark contrast with the known results for the map $\mathcal{X}_g\colon \rr_g\rightarrow \mm_{2g-1}$ for which we know that transversality with Brill-Noether loci is not always satisfied (see \cite[Theorem 0.4]{FarAprGreencong} and \cite[Theorem 1.4]{Bertram}).

\textbf{Proof of Theorem \ref{kodairarr}:} That $\rr_{g,2}$ is of general type for $g\geq 22$ follows from the fact that $\mm_{g,2/\mathbb{Z}_2}$ is of general type in this range (see \cite{EisenbudHarrisg>23} and \cite{FarJensenPayn} and recall that the second symmetric power of a generic curve of genus $g\geq 3$ is of general type). We are left to treat the cases $g=13$ and $16\leq g \leq 21$. As we do not have a uniform way of finding divisors satisfying all the slope requirements, the cases $g=13$ and $16\leq g\leq 21$ will be proven one by one. 

In order to prove Theorem \ref{kodairarr} we set up some notations. We denote by $\pi_g\colon\rr_g \rightarrow \mm_g$ the map forgetting the Prym structure. Breaking with the convention, we will denote in this proof by $\pi_{g,2}$ the forgetful map from $\rr_{g,2}$ to $\mm_g$.   

$\bullet$ For $g = 13$ we consider the following classes of divisors (up to multiplication by some positive constant): 
\begin{enumerate}
	\item The pullback $\chi_{13,2}^*[\mm^2_{26, 19}]$ of the Brill-Noether divisor $\mm^2_{26, 19}$ on $\mm_{26}$: 
	\[ \chi_{13,2}^*[\mm^2_{26, 19}] = 29\psi + 464\lambda - 94\delta_0^{\mathrm{ram}} - 72\delta_0' - 72\delta_{0:13,\left\{\eta\right\}}-\cdots  \]
	\item The pullback $\pi_{13,2}^{*}[\mm^1_{13,7}]$ of the Brill-Noether divisor $\mm^1_{13,7}$ on $\mm_{13}$: 
	\[ \pi_{13,2}^{*}[\mm^1_{13,7}] = 48\lambda - 14\delta_0^{\mathrm{ram}} - 7\delta_0'-\cdots \]
	\item The pullback $ i^*_{13,2}[\cU_{14,4}]$ of the Prym-Koszul divisor $\cU_{14,4}$ on $\rr_{14}$ (see \cite[Theorem 0.6]{FarLud}): 
	\[ i^*_{13,2}[\cU_{14,4}] = 21\psi + 180\lambda -42\delta^{\mathrm{ram}}_0 - 28 \delta_0' - \alpha\delta_{0:13,\left\{\eta\right\}}- \cdots \]
where we get from Remark \ref{description of i} and \cite[Proposition 1.9]{FarLud} that $\alpha \geq 100$.
\end{enumerate}
 We consider the sum 
\[ \frac{1}{92}\chi^*_{13,2}[\mm^2_{26, 19}] + \frac{1}{23}\pi_{13,2}^{*}[\mm^1_{13,7}] + \frac{3}{92}i^*_{13,2}[\cU_{14,4}] = \psi + 13\lambda - 2\delta_0' - 3\delta_0^{\mathrm{ram}}- \cdots \]
We still need to check that this divisor satisfies the slope requirements for the other boundary divisors. The coefficient of $\delta_{0:13,\left\{\eta\right\}}$ is greater than or equal to 
$ \frac{1}{92}\cdot 72 + 0 + \frac{3}{92}\cdot 100 > 3$. For the boundary divisor $\delta_{0:13,\left\{\OO\right\}}$ the coefficient is greater than or equal to
$ 2\cdot \frac{\textrm{coefficient}_{\delta_1}[\mm^2_{26, 19}]}{92} = 2\cdot \frac{200}{92} > 3$.

It is easy to check that all other slope requirements are respected. Consequently $K_{\rr_{13,2}}$ is in the effective cone and $\rr_{13,2}$ has non-negative Kodaira dimension.

$\bullet$ For the space $\rr_{16,2}$ we consider the following divisors (up to multiplication by a positive constant): 
\begin{enumerate}
\item The pullback $\chi_{16,2}^*[\mm^2_{32,23}]$ of the Brill-Noether divisor $\mm^2_{32,23}$ on $\mm_{32}$: 
\[ \chi_{16,2}^*[\mm^2_{32,23}] = 35\psi + 560\lambda -114\delta_0^{\mathrm{ram}}-88\delta_0'- 88\delta_{0:16,\left\{\eta\right\}} - \cdots \]
\item The pullback $\pi_{16,2}^*[\zz_{16,1}]$ of the Koszul divisor $\zz_{16,1}$ (see \cite[Theorem 1.1]{Farkosz}) on $\mm_{16}$:
\[ \pi_{16,2}^*[\zz_{16,1}] = 407\lambda - 122\delta_0^{\mathrm{ram}} - 61\delta_0' -\cdots \] 
\item The pullback $i_{16,2}^*\circ \pi_{17}^*[\mm^1_{17,9}]$ of the Brill-Noether divisor $\mm^1_{17,9}$ on $\mm_{17}$: 
\[ i_{16,2}^*\circ \pi_{17}^*[\mm^1_{17,9}] = 3\psi + 20\lambda - 6\delta_0^{\mathrm{ram}} - 3\delta_0' - 16\delta_{0:16,\left\{\eta\right\}} - \cdots \] 	
\end{enumerate}
For small enough $\epsilon$, we consider the sum 
\[ \text{\large(}\frac{62}{4933}+\frac{1}{4993}\epsilon\text{\large)}\chi_{16,2}^*[\mm^2_{32,23}] + \text{\large(}\frac{27}{4993}+\frac{80}{4993}\epsilon\text{\large)}  \pi_{16,2}^*[\zz_{16,1}] + \text{\large(}\frac{921}{4933}-\frac{1656}{4933}\epsilon\text{\large)} i_{16,2}^*\circ \pi_{17}^*[\mm^1_{17,9}]\]
which is equal to
\[ (1-\epsilon)\psi + 13\lambda - 2\delta_0' - \text{\large(}\frac{15888}{4933}-\frac{62}{4933}\epsilon\text{\large)}\delta^{\mathrm{ram}}_0 - \text{\large(}\frac{20192}{4933}-\frac{26408}{4933}\epsilon\text{\large)}\delta_{0:16,\left\{\eta\right\}}-\cdots \]
 
Checking that all other slope requirements are satisfied is immediate. The conclusion follows because we know that $\psi$ is big and nef (see \cite[Proposition 1.2]{FarVerUnivtheta}).

$\bullet$ For the space $\rr_{17,2}$ we consider the following divisors (up to multiplication by a positive constant): 
\begin{enumerate}
\item The pullback $\chi_{17,2}^*[\mm^4_{34, 31}]$ of the Brill-Noether divisor $\mm^4_{34, 31}$ on $\mm_{34}$: 
\[ \chi_{17,2}^*[\mm^4_{34, 31}] = 111\psi + 1776\lambda - 362\delta_0^{\mathrm{ram}} - 280\delta_0' - 280\delta_{0:17,\left\{\eta\right\}} -\cdots \] 
\item The pullback $\pi^*_{17,2}[\mm^1_{17,9}]$ of the Brill-Noether divisor $\mm^1_{17,9}$ on $\mm_{17}$: 
\[ \pi^*_{17,2}[\mm^1_{17,9}] = 20\lambda - 6\delta_0^{\mathrm{ram}} - 3\delta_0' - \cdots \]
\item The pullback $i_{17,2}^*\circ \pi_{18}^*[\gp^5_{18,20}]$ of the Gieseker-Petri divisor $\gp^5_{18,20}$ on $\mm_{18}$: 
\[   i_{17,2}^*\circ \pi_{18}^*[\gp^5_{18,20}] = 77\psi + 516 \lambda - 154\delta_0^{\mathrm{ram}} - 77\delta_0' - 408\delta_{0:17,\left\{\eta\right\}} -\cdots  \]	
\end{enumerate}
For $\epsilon$ small enough, we consider the sum
\[ \text{\large(}\frac{85}{21832}-\frac{1}{2729}\epsilon\text{\large)} \chi_{17,2}^*[\mm^4_{34, 31}] + \text{\large(}\frac{2489}{21832}+\frac{966}{2729}\epsilon\text{\large)}\pi^*_{17,2}[\mm^1_{17,9}] + \text{\large(}\frac{161}{21832}-\frac{34}{2729}\epsilon\text{\large)}i_{17,2}^*\circ \pi_{18}^*[\gp^5_{18,20}] \]
which is equal to 
\[ (1-\epsilon)\psi + 13\lambda - 2\delta_0' - \text{\large(}\frac{70498}{21832}+\frac{198}{2729}\epsilon\text{\large)} \delta_0^{\mathrm{ram}}- \cdots  \]
and it can be checked that it respects all the slope requirements. The divisor $i_{17,2}^*\circ \pi_{18}^*[\gp^5_{18,20}]$ is necessary here for the coefficient of $\delta_{0:17,\left\{\eta\right\}}$. 

As the situation is entirely similar for all the other cases, we will simply state which divisors are used, skipping the numerical details. 
\begin{itemize}
	\item  For the space $\rr_{18,2}$ we will use the divisors $\chi^*_{18,2}[\gp^5_{36,35}]$, $\pi_{18,2}^*[\gp^5_{18,20}]$ and $i_{18,2}^*\circ \pi_{19}^*[\mm^1_{19,10}]$. 
	\item  For the space $\rr_{19,2}$ we will use the divisors $\chi_{19,2}^*[\mm^2_{38,27}]$, $\pi_{19,2}^*[\mm^1_{19,10}]$ and $i_{19,2}^*\circ \pi_{20}^*[\mm^2_{20,15}]$.
	\item For the space $\rr_{20,2}$ we will use the divisors
	$\chi_{20,2}^*[\gp^7_{40,42}]$, $\pi_{20,2}^*[\mm^2_{20,15}]$ and  $i_{20,2}^*\circ \pi_{21}^*[\mm^1_{21,11}]$. 
	\item  For the space $\rr_{21,2}$ we will use the divisors $\chi_{21,2}^*[\gp^6_{42,42}]$, $\pi_{21,2}^*[\mm^1_{21,11}]$ and $i_{21,2}^*\circ \pi_{22}^*[\gp^{11}_{22,12}]$. 
\end{itemize}

Using the outlined divisors, we conclude that $K_{\rr_{g,2}}$ is big for $g\geq 16$. The proof follows as a consequence of Theorem \ref{pluricanonical extension} . 

\hfill $\square$ 

Next, we will study the birational geometry of $\cR_{g,2}$ when $g$ is small. 

\textbf{Proof of Theorem \ref{lowgenus}:} For the length $g+2$ partition $\mu = (1,1,2,\ldots,2,2g-4)$, we consider the irreducible stratum 
\[ \mathcal{Q}_g(\mu) = \left\{[C,x,y, z_1,\ldots,z_g]\in \cM_{g,g+2} \ | \ \OO_C(x+y+2\sum_{i=1}^{g-1}z_i + (2g-4)z_g)\cong \omega_C^{\otimes 2}\right\} \]
and the map $\mathcal{Q}_g(\mu) \rightarrow \mathcal{R}_{g,2}$ defined as: 
\[ [C,x,y,z_1,\ldots,z_g]\mapsto [C,x+y, \omega_C(-x-y-\sum_{i=1}^{g-1}z_i-(g-2)z_g)]. \]
This map is dominant and $\mathcal{Q}_g(\mu)$ is uniruled for $3\leq g\leq 6$, see \cite[Theorem 0.3]{BAR18}. This concludes the proof. \hfill $\square$ 

It is important to note that the Prym moduli spaces provide an interesting geometric property for the divisors of "small" slope on $\mm_g$.  

\begin{prop}
	Let $D$ be an effective divisor on $\cM_g$ of slope $s(D) < 10$ and $\overline{D}$ its closure in $\mm_g$. Depending on the parity of $g$ we have the following: 
	
	i. If $g = 2i+1$, then $\overline{D}$ contains the locus $\chi_{i+1}(\Delta_{1:i})$,
	
	ii. If $g =2i$, then $\overline{D}$ contains the locus $\chi_{i,2}(\Delta_{i-1:1,\left\{\eta\right\}})$.
\end{prop} 

\begin{proof} We start with the case $g = 2i+1$. If $\chi_{i+1}(\rr_{i+1})$ is contained in $\overline{D}$, the conclusion is clear; hence we can assume the contrary. 
	
We consider a generic pencil of elliptic curves. We attach a base point to a generic point of a generic curve of genus $i$ and obtain in this way a test curve $A$ in $\mm_{i+1}$. We denote by $A_{1:i}$ the pullback of this test curve to $\Delta_{1:i}\subseteq\rr_{i+1}$. We know from \cite{Bud-adm}, \cite{Carlos} that 
\[ A_{1:i} \cdot \lambda = 3, \ A_{1:i} \cdot \delta_0' = 12, \ A_{1: i} \cdot \delta^{\mathrm{ram}}_0 = 12, \ A_{1:i} \cdot \delta_{1:i} = -3 \] 	
while the intersection with all other boundary divisors is $0$. 	

We have that $\chi_{i+1}^*[\overline{D}]$ is an effective divisor in $\rr_{i+1}$. If this divisor does not contain $\Delta_{1:i}$ in its support, it follows that $\chi_{i+1}^*[\overline{D}]\cdot A_{1:i} \geq 0$. If we write (up to multiplication by a constant) the class of $\overline{D}$ as $s\lambda - \delta_0 -\cdots$ the inequality $\chi_{i+1}^*[\overline{D}]\cdot A_{1:i} \geq 0$ becomes: 
\[ 2s - 4(2 + 1+\frac{s}{4}) + 2 \geq 0 \]
that is $s\geq 10$. Hence our assumption was wrong and we get the conclusion for the odd case. 

For the case $g = 2i$ we can define a test curve 
$A_{i-1:1,\left\{\eta\right\}}$ on $\rr_{i,2}$ by considering two points on the genus $i-1$ component of the test curve $A$ in $\mm_i$ and pulling it back to $\Delta_{i-1:1,\left\{\eta\right\}}$. We have the intersection numbers: 
\[ A_{i-1:1,\left\{\eta\right\}}\cdot \lambda = 3, \   A_{i-1:1,\left\{\eta\right\}}\cdot \delta_0' = 12, \  A_{i-1:1,\left\{\eta\right\}} \cdot \delta^{\mathrm{ram}}_0 = 12, \  A_{i-1:1,\left\{\eta\right\}} \cdot \delta_{i-1:1,\left\{\eta\right\}} = -3 \] 
while the intersection of $ A_{i-1:1,\left\{\eta\right\}}$ with $\psi$ and all other boundary classes is 0. By considering the pullback of $\chi_{i,2}\colon\rr_{i,2} \rightarrow \mm_{2i}$, the proof follows analogously to the case $g= 2i+1$.
\end{proof} 

\section{The singularities of $\rr_{g,2}$}
In order to conclude Theorem \ref{kodairarr} we still need to prove that any pluricanonical form on the smooth locus of $\rr_{g,2}$ can be holomorphically extended to any desingularisation. Namely, we need to show that:  
\begin{trm} \label{pluricanonical extension}
	We fix $g\geq 4$ and let $\widehat{\cR}_{g,2} \rightarrow \rr_{g,2}$ be any desingularisation. Then every pluricanonical form defined on the smooth locus $\rr^{\textrm{reg}}_{g,2}$ extends holomorphically to the space $\widehat{\cR}_{g,2}$.
\end{trm} 

Our approach in proving this statement follows closely the ones in \cite{Ludspin} and \cite{FarLud}. As the proofs will be very similar to those appearing in these papers, our main goal will be to point out the differences in the statements. Next, we want to describe the smooth locus of $\rr_{g,2}$ and for this we give the following definitions. 

\begin{defi} An irreducible component $C_j$ of a quasistable curve $[X,x+y]$ is called a rational tail if the arithmetic genus $p_a(C_j)$ is 0 and $C_j \cap \overline{X\setminus C_j} = \left\{p\right\}$. The node $p$ is then called a rational tail node. A non-trivial automorphism $\sigma$ of $[X, x+y]$ is called a rational tail automorphism (with respect to $C_j$) if $\sigma_{X\setminus C_j}$ is the identity.  
\end{defi}
It is clear from the definition that if $C_j$ is a rational tail then the two points $x$ and $y$ are on $C_j$. With the obvious modifications, we can define what it means for a morphism $\sigma$ to be an elliptic tail automorphism (with respect to an elliptic tail $C_j$), with the remark that in the definition we add the extra condition $x, y \notin C_j$. 

\begin{defi}
	An exceptional component $E$ of a quasistable curve $[X, x+y]$ is called a disconnecting exceptional component if $\overline{X\setminus E}$ consists of two disjoint connected components, which we denote $X_1$ and $X_2$.
	
	Let $[X, x+y, \eta, \beta]$ a $2$-branched Prym curve and $E$ a disconnecting exceptional component of $[X, x+y]$. We denote by $\gamma_E \in \textrm{Aut}_0(X,x+y,\eta, \beta)$ the inessential automorphism that is the multiplication with 1 and respectively $-1$ in every fiber of $\eta$ over $X_1$ and respectively $X_2$. 
\end{defi} 

Similarly to Theorem 6.5 in \cite{FarLud} and Proposition 2.15 in \cite{Ludspin} we get the following theorem: 
\begin{trm}
	Let $(X, x+y, \eta, \beta)$ be a $2$-branched Prym curve of genus $g\geq 4$. Then the point $[X, x+y, \eta, \beta]$ in $\rr_{g,2}$ is smooth if and only if $\mathrm{Aut}(X, x+y, \eta, \beta)$ is generated by rational tail involutions, elliptic tail involutions and automorphisms of the form $\gamma_E$ for some disconnecting exceptional component $E$.  
\end{trm} 
\begin{proof}
	Locally at $[X, x+y, \eta, \beta]$ the coarse moduli space $\rr_{g,2}$ is given as a neighbourhood of 0 in the quotient $\mathbb{C}^{3g-1}_\tau/\mathrm{Aut}(X, x+y, \eta, \beta)$. 
	
	In this situation, we know from \cite{prill} that $[X, x+y, \eta, \beta]$ is smooth in $\rr_{g,2}$ if and only if the group $\mathrm{Aut}(X, x+y, \eta, \beta)$ is generated by quasi-reflections. Arguing as in \cite[Proposition 6.6]{FarLud} we conclude that the only automorphisms acting as quasi-reflections are those appearing in the statement.  
\end{proof}

The non-canonical singularities of the space $\rr_{g,2}$ can be easily described. 

\begin{trm} \label{canonicalsingularities}
	Let $g\geq 4$. Then $[X, x+y, \eta, \beta]$ is a non-canonical singularity if and only if $X$ has an elliptic tail $C_j$ with $j$-invariant 0 and $\eta$ is trivial on $C_j$.
\end{trm}
We observe that this description is very similar to other cases in the literature, see \cite{KodMg}, \cite{Logan},  \cite{Ludspin}, \cite{FarLud} and \cite{FarVerUnivJac}. The approach in proving Theorem \ref{canonicalsingularities} will be similar to the one in \cite{FarLud} and \cite{Ludspin} and we will continue by pointing out the differences in our case. First, we will define what it means for a pair $\text{\large (}(X,x+y,\eta, \beta), \sigma\text{\large)}$, where $\sigma \in \textrm{Aut}(X, x+y, \eta, \beta)$, to be singularity reduced. 

For such a pair $\text{\large (}(X,x+y,\eta, \beta), \sigma\text{\large)}$ we define $(C,x+y)$ to be the stable model of $(X,x+y)$ and $\sigma_C$ to be the automorphism of $C$ induced by $\sigma$. We consider the definition:

\begin{defi}
	Let $p_{i_0}, p_{i_1} = \sigma_C(p_{i_0}), \ldots, p_{i_{m-1}} = \sigma^{m-1}_C(p_{i_0})$ be distinct nodes of $C$, cyclically permuted by $\sigma_C$ and $p_{i_0}$ is not a disconnecting exceptional node, a rational tail node or an elliptic tail node. Then $\sigma$ acts on the subspace $\bigoplus_{j=0}^{m-1}\mathbb{C}_{\tau_{i_j}} \subseteq \mathbb{C}_\tau^{3g-1}$ as $\sigma\cdot \tau_{i_{j-1}} = c_j\tau_{i_j} \ \forall j =\overline{0,m-1}$ for some constants $c_j$. We say the pair $\text{\large (}(X,x+y, \eta, \beta), \sigma\text{\large)}$ is singularity reduced if for every cycle as above we have $\prod_{j=1}^{m}c_j \neq 1$.
\end{defi}

Arguing as in \cite[Proposition 6.8]{FarLud} and \cite[Proposition 3.6]{Ludspin} it is sufficient to prove that if the pair $\text{\large (}(X,x+y, \eta, \beta), \sigma\text{\large)}$ is singularity reduced and satisfies $\textrm{age}(\sigma, \xi_n) <1$ then $X$ has an elliptic tail as in Theorem \ref{canonicalsingularities}. As in \cite{FarLud} and \cite{Ludspin} we denote by $(*)$ the assumption that $\text{\large (}(X,x+y, \eta, \beta), \sigma\text{\large)}$ is singularity reduced and satisfies $\textrm{age}(\sigma, \xi_n) <1$. We then have that: 
\begin{prop}
	If $(*)$ holds, then $\sigma_C$ fixes all nodes and all components of the stable model $(C, x+y)$ of $(X, x+y)$. 
\end{prop} 

Next we will look at how $\sigma$ acts on the components of $C$ and describe all possible situations where the age contribution of the respective component is less than 1. This description is similar to that in \cite[Proposition 6.12]{FarLud} with some extra cases coming from the existence of the two marked points. 

\begin{prop} 
	Assume $(*)$ holds and let $C_j$ a component of $C$ with normalization denoted $C_j^\nu$. We denote by $D_j$ the divisor of marked points on $C_j^\nu$ and $\varphi_j \coloneqq \sigma^\nu_{|C^\nu_j}$. Then $(C^\nu_j, D_j, \varphi_j)$ is of one of the following types and the contribution to $\textrm{age}(\sigma, \xi_n)$ coming from $H^1(C^\nu_j, T_{C_j^\nu}(-D_j)) \subseteq \mathbb{C}_v^{3g-1}$ is at least the quantity $w_j$: \\
i. Identity component $\varphi_j = \textrm{Id}_{C^\nu_j}$, the pair $(C^\nu_j, D_j)$ is arbitrary and $w_j = 0$. \\
ii. Elliptic tail: $C^\nu_j$ is elliptic, $D_j = p_1$ and $p_1$ is fixed by $\varphi_j$. Depending on the order of $\varphi_j$ we distinguish the subcases: \par 
a. $\mathrm{ord}(\varphi_j) = 2$ and $w_j = 0$ \par 
b. $\mathrm{ord}(\varphi_j) = 4$, $C^\nu_j$ has $j$-invariant $1728$ and $w_j = \frac{1}{2}$ \par 
c. $\mathrm{ord}(\varphi_j) = 3$ or $6$, $C^\nu_j$ has $j$-invariant 0 and $w_j = \frac{1}{3}$ \\
iii. Elliptic ladder: $C_j^\nu$ is elliptic, $D_j = p_1 + p_2$ with both markings coming from nodes of $C$ and $\varphi_j$ fixes $p_1$ and $p_2$. We distinguish three subcases depending on the order of $\varphi_j$: \par 
a. $\mathrm{ord}(\varphi_j) = 2$ and $w_j =\frac{1}{2}$ \par 
b. $\mathrm{ord}(\varphi_j) = 4$, $C_j^\nu$ has $j$-invariant $1728$ and $w_j =\frac{3}{4}$ \par 
c. $\mathrm{ord}(\varphi_j) = 3$, $C_j^\nu$ has $j$-invariant $0$ and $w_j =\frac{2}{3}$ \\ 
iv. Hyperelliptic tail: $C^\nu_j$ has genus $2$, $\varphi_j$ is the hyperelliptic involution, $D_j$ is of the form $D_j = p_1$ with $p_1$ fixed by $\varphi_j$ and $w_j = \frac{1}{2}$ \\ 
v. Rational tail: $C^\nu_j$ is rational, $D_j = p_1 + x+y$, $\mathrm{ord}(\varphi_j) = 2$, the point $p_1$ is fixed by $\varphi_j$ while $x$ and $y$ are permuted, and $w_ j =0$. \\
vi. Rational ladder: $C^\nu_j$ is rational, $D_j = p_1 + p_2+ x+y$, $\mathrm{ord}(\varphi_j) = 2$, the points $p_1, p_2$ are fixed by $\varphi_j$ while $x$ and $y$ are permuted, and $w_ j =\frac{1}{2}$. \\ 
vii. $1$-pointed elliptic tail: $C_j^\nu$ is elliptic, $D_j = p_1 + p_2$ where $p_1$ comes from a node of $C$ and $p_2$ comes from one of the markings $x, y$. Both $p_1$ and $p_2$ are fixed by $\varphi_j$. We distinguish three subcases depending on the order of $\varphi_j$: \par 
a. $\mathrm{ord}(\varphi_j) = 2$ and $w_j = \frac{1}{2}$ \par 
b. $\mathrm{ord}(\varphi_j) = 4$, $C_j^\nu$ has $j$-invariant $1728$ and $w_j = \frac{3}{4}$ \par 
c. $\mathrm{ord}(\varphi_j) = 3$ or $6$, $C_j^\nu$ has $j$-invariant $0$ and $w_j = \frac{2}{3}$ \\ 
viii. $2$-pointed elliptic tail: $C_j^\nu$ is elliptic, $D_j = p_1 +x+y$ with $x$ and $y$ permuted by $\varphi_j$. Again, we distinguish two subcases depending on the order of $\varphi_j$: \par 
a. $\mathrm{ord}(\varphi_j) = 2$ and $w_j = \frac{1}{2}$ \par 
b. $\mathrm{ord}(\varphi_j) = 6$, $C_j^\nu$ has $j$-invariant $0$ and $ w_j = \frac{1}{3} + \frac{1}{3} +\frac{1}{6} = \frac{5}{6}$.
\end{prop} 

If $(*)$ holds, the cases where $w_j > \frac{1}{3}$ cannot appear, while if for every irreducible component of $C$ we have $w_j$ = 0, we get that $\sigma$ is a composition of quasi-reflections. This implies Theorem \ref{canonicalsingularities}. In order to extend the  pluricanonical forms over the locus of non-canonical singularities, the method outlined in \cite[pages $41-44$]{KodMg} works in our situation, hence Theorem \ref{pluricanonical extension} follows. 
 
\section{The Prym-canonical divisorial strata} 

We consider the moduli space $\mathcal{C}^n\mathcal{R}_g$ parametrizing isomorphism classes of pairs $(X, x_1,\ldots,x_n, \eta)$ where $(X,x_1,\ldots,x_n)$ is an $n$-pointed smooth curve of genus $g$ and $\eta$ is a non-trivial $2$-torsion of $\OO_X$ in $\mathrm{Pic}(X)$. 

For a partition $\underline{d} = (d_1,\ldots,d_n)$ of $g-1$, we consider the Prym-canonical divisorial stratum $PD_{\underline{d}}$, defined as the locus 
\[ PD_{\underline{d}} \coloneqq \left\{ [X,x_1,\ldots, x_n,\eta] \in \mathcal{C}^n\mathcal{R}_g \ | \ h^0\text{\large(}X, \omega_X\otimes\eta(-\sum_{i=1}^nd_ix_i)\text{\large)} \geq 1 \right\}\] 
Our goal is to provide a suitable compactification $\cpp_g$ of $\mathcal{C}^n\mathcal{R}_g$ and compute the class $[\overline{PD}_{\underline{d}}]$ in $\textrm{Pic}(\cpp_g)$. 

Using the same approach as in \cite{corn}, we compactify $\pp_g$ to a moduli space $\cpp_g$ parametrizing isomorphism classes of pairs $(X, x_1,\ldots,x_n, \eta, \beta)$ where $(X, x_1,\ldots,x_n)$ is a quasistable $n$-pointed curve of genus $g$ and $\beta\colon \eta^{\otimes 2} \rightarrow \OO_X$ is a homomorphism of invertible sheaves that satisfies the properties: \par
1. The line bundle $\eta$ has total degree $0$ on $X$ and degree $1$ on every exceptional component, \par
2. The morphism $\beta$ is generically non-zero away from the exceptional components.  

The notion of isomorphism is simply the pointed generalization of the one considered in \cite{Casa}. 

Next, we consider the irreducible boundary divisors of $\cpp_g$ and describe a generic element for each one of them. They are as follows: 
\begin{itemize}
\item The divisors $\Delta_0', \Delta_0''$ and $\Delta_0^{\mathrm{ram}}$ whose generic point corresponds to an element $[X, x_1,\ldots,x_n, \eta, \beta]$ where $[X,\eta, \beta] \in \rr_g$ is generic in $\Delta_0', \Delta_0''$ and respectively $\Delta_0^{\mathrm{ram}}$ and $x_1,\ldots,x_n\in X$ are generic points of the non-exceptional component of $X$. \par
\item The divisors $\Delta_{g,S}$ for $S\subseteq\left\{1,\ldots,n\right\}$ and $|S|\leq n-2$. A generic element $[X,x_1,\ldots,x_n,\eta,\beta]$ satisfies that $[X,x_1,\ldots,x_n]\in \Delta_{0,S^c}\subseteq \mm_{g,n}$, the line bundle $\eta$ is trivial on the rational component and is a non-trivial $2$-torsion on the genus $g$ component. \par
\item The divisors $\Delta_{i,S:g-i}$ for $1\leq i\leq g-1$ and $S\subseteq \left\{1,\ldots,n\right\}$, whose generic element $[X, x_1,\ldots,x_n, \eta, \beta]$ satisfies that $[X, \eta, \beta]$ is in $\Delta_{i:g-i}\subseteq \rr_g$ and $[X,x_1,\ldots,x_n] \in \Delta_{i,S}\subseteq \mm_{g,n}$. We remark that the notations $\Delta_{i, S:g-i}$ and $\Delta_{g-i,S^c:i}$ refer to the same divisor of $\cpp_g$. \par
\item The divisors $\Delta_{i,S}$ for $1\leq i\leq g-1$ and $S\subseteq \left\{1,\ldots,n\right\}$, whose generic element $[X, x_1,\ldots,x_n, \eta, \beta]$ satisfies that $[X, \eta, \beta]$ is in $\Delta_i\subseteq \rr_g$ and $\left\{x_i\right\}_{i\in S}$ are generic points of the component of $X$ on which $\eta$ is non-trivial. 
\end{itemize}
Next, we consider maps between moduli spaces of pointed (Prym) curves and we describe the action of the pullback at the level of Picard groups. This will allow us to compute the classes of the Prym-canonical divisorial strata. 
 
We consider the map $\pi\colon \cpp_g \rightarrow \mm_{g,n}$ forgetting the Prym structure and stabilizing the underlying $n$-pointed curve. We denote $\psi_j \coloneqq \pi^{*}\psi_j$, $\lambda \coloneqq \pi^{*}\lambda$ in $\mathrm{Pic}(\mathrm{\cpp_g})$ and we want to describe the pullback of the classes $\lambda, \psi_j, \delta_0', \delta_0'', \delta_0^{\mathrm{ram}}$, $\delta_{i,S:g-i}$ and $\delta_{i,S}$ with respect to different maps. 
\begin{prop}\label{pullmap1} Consider the map $\pi_1\colon \mm_{g-i,n+1-s} \rightarrow \cpp_g$ defined as 
	\[[C,x_1,\ldots,x_{n-s},x] \mapsto [C\cup_{x\sim y}Y,x_1,\ldots,x_n, \OO_C, \eta_Y]\] where $[Y,y,x_{n-s+1},\ldots,x_n,\eta_Y]$ is a generic element in $\mathcal{C}^{s+1}\mathcal{R}_i$. The pullback at the level of Picard groups $\pi_1^*\colon \mathrm{Pic}(\cpp_g) \rightarrow \mathrm{Pic}(\mm_{g-i,n+1-s})$ satisfies: 
\[ \pi_1^{*}\lambda = \lambda, \ \pi_1^*\delta_0'' = \pi_1^*\delta_0^{\mathrm{ram}} = 0, \ \pi_1^*\delta_0' = \delta_0 \] 
\[\pi_1^*\psi_j = \psi_j \ \mathrm{for} \ 1\leq j\leq n-s, \ \pi_1^*\psi_j =0 \ \mathrm{for} \ j\geq n-s+1\]
\[ \pi_1^*\delta_{j,S:g-j} =0 \ \textup{for every} \ 1\leq j\leq g-1 \ \textup{and every} \ S\subseteq \left\{1,\ldots, n\right\} \]
For $T = \left\{n-s+1,\ldots, n\right\}$ we have: 
 \[ \pi_1^*\delta_{j,S} = 
\begin{cases} 
\delta_{j-i, (S\setminus T)\cup\left\{n-s+1\right\}} & \mathrm{when} \ i\leq j\leq g, \ T\subseteq S \ \mathrm{and} \ (j,S)\neq (i,T) \\
-\psi_{n-s+1} & \mathrm{when} \ j=i \ \mathrm{and} \ T=S \\
0 & \textup{otherwise}.
\end{cases}
\]

\end{prop}  
\begin{proof}
	The equalities with right term 0 follow because $\mathrm{Im}(\pi_1)$ does not intersect the respective divisors. For the other equalities we look at the composition map $\pi\circ \pi_1\colon \mm_{g-i,n+1-s} \rightarrow \mm_{g,n}$ and we use that $\pi_1^*\circ \pi^* = (\pi\circ \pi_1)^*$. The description of $(\pi\circ \pi_1)^*$ appearing in \cite{cornintersection} implies the conclusion.
\end{proof}

Similarly we get: 

\begin{prop} \label{pullmap2} Let $\pi_2\colon \cplow_{g-i}\rightarrow \cpp_g$ be given as \[[C,x_1,\ldots,x_{n-s},x, \eta_C] \mapsto [C\cup_{x\sim y}Y,x_1,\ldots,x_n, \eta_C, \eta_Y]\] 
where $[Y,y,x_{n-s+1},\ldots,x_n,\eta_Y]$ is a generic point of $\mathcal{C}^{s+1}\mathcal{R}_i$. The pullback $\pi_2^*\colon \mathrm{Pic}(\cpp_g) \rightarrow \mathrm{Pic}(\cplow_{g-i})$ satisfies: 
\[ \pi_2^{*}\lambda = \lambda, \ \pi_2^*\delta_0'' = 0, \ \pi_2^*\delta_0' = \delta_0'+\delta_0'', \ \pi_2^*\delta_0^{\mathrm{ram}} = \delta_0^{\mathrm{ram}} \] 
\[\pi_2^*\psi_j = \psi_j \ \mathrm{for} \ 1\leq j\leq n-s, \ \pi_2^*\psi_j =0 \ \mathrm{for} \ j\geq n-s+1\]
For $T = \left\{n-s+1,\ldots, n\right\}$ we have: 
 \[ \pi_2^*\delta_{j,S} = 
\begin{cases} 
\delta_{j-i, (S\setminus T)\cup\left\{n-s+1\right\}} & \mathrm{when} \ i+1\leq j\leq g \ \mathrm{and} \ T\subseteq S \\
0 & \textup{otherwise}.
\end{cases}
\]
We recall that the divisor $\Delta_{j,S:g-j}$ admits the alternative notation $\Delta_{g-j,S^c:j}$. We choose the one where the set contains $\left\{n-s+1\right\}$ and we have: 
\[ \pi_2^*\delta_{j,S:g-j} = 
\begin{cases} 
\delta_{j-i, (S\setminus T)\cup\left\{n-s+1\right\}:g-j} + \delta_{g-j,((S\setminus T)\cup\left\{n-s+1\right\})^c} & \mathrm{when} \ i\leq j\leq g-1, \ T\subseteq S \ \mathrm{and} \ (j,S)\neq (i,T) \\ 
-\psi_{n-s+1} & \mathrm{if} \ i = j \ \mathrm{and} \ S= T \\
0 & \textup{otherwise}.
\end{cases}
\]
\end{prop}
\begin{proof}
    We can check using test curves that the boundary classes are linearly independent in $\mathrm{Pic}(\cpp_g)$.  
	Using this and the obvious commutative diagram  
		\[
	\begin{tikzcd}
	\cplow_{g-i} \arrow{r}{\pi_2}  \arrow[swap]{d}{} & \cpp_{g} \arrow{d}{} \\
	\mm_{g-i,n+1-s} \arrow{r}{}& \mm_{g,n}
	\end{tikzcd}
	\] 
	the conclusion follows.
\end{proof} 
Lastly, we have: 
\begin{prop} \label{pullmap3}
	 Let $\pi_3\colon \cplow_{g-i}\rightarrow \cpp_g$ be given as 
	 \[[C,x_1,\ldots,x_{n-s},x, \eta_C] \mapsto [C\cup_{x\sim y}Y,x_1,\ldots,x_n, \eta_C, \OO_Y]\] 
	 where $[Y,y,x_{n-s+1},\ldots,x_n]$ is a generic element of $\cM_{i,s+1}$. The pullback $\pi_3^*\colon \mathrm{Pic}(\cpp_g) \rightarrow \mathrm{Pic}(\cplow_{g-i})$ satisfies: 
	 \[ \pi_3^{*}\lambda = \lambda, \ \pi_3^*\delta_0'' = \delta_0'', \ \pi_3^*\delta_0' = \delta_0', \ \pi_3^*\delta_0^{\mathrm{ram}} = \delta_0^{\mathrm{ram}} \] 
	 \[\pi_3^*\psi_j = \psi_j \ \mathrm{for} \ 1\leq j\leq n-s, \ \pi_3^*\psi_j =0 \ \mathrm{for} \ j\geq n-s+1\]
	 We denote again $T = \left\{n-s+1,\ldots, n\right\}$ and we have: 
	  \[ \pi_3^*\delta_{j,S} = 
	 \begin{cases} 
	 \delta_{j-i, (S\setminus T)\cup\left\{n-s+1\right\}} & \mathrm{when} \ i+1\leq j\leq g \ \mathrm{and} \ T\subseteq S \\
	 \delta_{j,S} & \mathrm{when} \ i+j\leq g, \ S\cap T =\emptyset \ \mathrm{and} \ (j,S)\neq (g-i, T^c) \\
	 -\psi_{n-s+1} & \mathrm{when} \ j=g-i \ \mathrm{and} \ S=T^c \\
	 0 & \textup{otherwise}.
	 \end{cases}
	 \]
	  \[ \pi_3^*\delta_{j,S:g-j} = 
	 \begin{cases} 
	 \delta_{j-i, (S\setminus T)\cup\left\{n-s+1\right\}:g-j} & \mathrm{when} \ i+1\leq j\leq g-1 \ \mathrm{and} \ T\subseteq S \\
	 \delta_{j,S:g-j-i} & \mathrm{when} \ i+j\leq g-1 \ \mathrm{and} \ S\cap T =\emptyset  \\
	 0 & \textup{otherwise}.
	 \end{cases}
	 \]
\end{prop}
 \begin{proof}
 	This follows analogously to Proposition \ref{pullmap1} and Proposition \ref{pullmap2}. 
 \end{proof}
 As we do not know if $\psi_1,\ldots, \psi_n$ and $\lambda$ generate $\mathrm{Pic}(\pp_g)$ we start by proving that the class $[PD_{\underline{d}}]$ is a linear combination of these classes in $\mathrm{Pic}(\pp_g)$. 
 
 \begin{prop} \label{smoothdivisors}Let $\underline{d} = (d_1,\ldots, d_n)$ be a partition of $g-1$ with all entries positive. We have the following equality in $\mathrm{Pic}(\pp_g)$:
 	\[ [PD_{\underline{d}}] = \sum_{j=1}^n \frac{d_j(d_j+1)}{2}\psi_j - \lambda \]
 \end{prop}
\begin{proof}
	We work over the locus $\cR^0_g$ of smooth Prym curves of genus $g$ without automorphisms. We consider the Cartesian diagram 
	
		\[
	\begin{tikzcd}
	\cC^{n+1} \coloneqq \pp_g^0\times_{\cR_g^0}\mathcal{C}^1\mathcal{R}_g^0 \arrow{r}{p_2}  \arrow[swap]{d}{p_1} & \mathcal{C}^1\mathcal{R}_g^0 \arrow{d}{p} \\
	\pp^0_g \arrow{r}{}& \cR_g^0
	\end{tikzcd}
	\]
Because the Prym curves have no automorphisms, it follows that there exists a line bundle $\mathcal{P}$ on $\mathcal{C}^1\mathcal{R}_g^0$ restricting to $\eta$ over each fiber $p^{-1}([X,\eta])$. We denote by $\Delta_i$ the diagonal of $\cC^{n+1}$ parametrizing points $[C,x,x_1,\ldots, x_n]$ satisfying $x=x_i$, and consider the short exact sequence 
\[ 0 \rightarrow p_2^*\mathcal{P} \rightarrow p_2^*\mathcal{P}\otimes \OO_{\cC^{n+1}}\text{\large (}\sum_{i=1}^{n}d_i\Delta_i\text{\large )} \rightarrow p_2^*\mathcal{P} \otimes \text{\Large (}\OO_{\cC^{n+1}}\text{\large (}\sum_{i=1}^{n}d_i\Delta_i     \text{\large )}
/\OO_{\cC^{n+1}}\text{\Large )} \rightarrow 0 \]
We pushforward this by $p_{1*}$ and obtain the exact sequence: 
\begin{multline*}
 0  \rightarrow 
 p_{1*}\text{\Large(}p_2^*\mathcal{P}\otimes \OO_{\cC^{n+1}}(\sum_{i=1}^{n}d_i\Delta_i) \text{\Large)} \rightarrow p_{1*}\text{\Large(}p_2^*\mathcal{P} \otimes \text{\large (}\OO_{\cC^{n+1}}(\sum_{i=1}^{n}d_i\Delta_i)/\OO_{\cC^{n+1}}\text{\large )}\text{\Large)}\xrightarrow{\alpha}
 R^1p_{1*}p_2^*\mathcal{P}\rightarrow\cdots \\ \cdots\rightarrow R^1p_{1*}\text{\Large(}p_2^*\mathcal{P}\otimes \OO_{\cC^{n+1}}\text{\large (}\sum_{i=1}^{n}d_i\Delta_i\text{\large )} \text{\Large)} \rightarrow 0
\end{multline*} 
We have that $p_{1*}\text{\Large(}p_2^*\mathcal{P} \otimes \text{\large (}\OO_{\cC^{n+1}}(\sum_{i=1}^{n}d_i\Delta_i)/\OO_{\cC^{n+1}}\text{\large )}\text{\Large)}$ is a vector bundle of rank $g-1$ with fiber over a point $[X,x_1,\ldots,x_n,\eta]\in \pp_g$ given as $H^0(X, \eta\otimes \text{\large (}\OO_X(\sum_{i=1}^nd_ix_i)/\OO_X)\text{\large )}$. Similarly $R^1p_{1*}p_2^*\mathcal{P}$ is a vector bundle of rank $g-1$ with fiber $H^1(X, \eta)$ over $[X,x_1,\ldots,x_n,\eta]$. The map $\alpha$ restricted to the fiber over $[X,x_1,\ldots,x_n, \eta]$ is the one induced by the exact sequence: 
\[ 0\rightarrow \eta \rightarrow \eta\text{\large(}\sum_{i=1}^nd_ix_i\text{\large)} \rightarrow \eta\otimes \text{\Large (}\OO_X(\sum_{i=1}^nd_ix_i)/\OO_X)\text{\Large )} \rightarrow 0 \]

The Riemann-Roch Theorem implies that $[PD_{\underline{d}}]$ is the degeneration locus of the map $\alpha$ and consequently
\[ [PD_{\underline{d}}] = -c_1\text{\Large(}p_{1!}\text{\large(}p_2^*\mathcal{P}\otimes \OO_{\cC^{n+1}}(\sum_{i=1}^{n}d_i\Delta_i)\text{\large)}\text{\Large )} \]

We apply the Grothendieck-Riemann-Roch formula and obtain 
\[ [PD_{\underline{d}}] = -p_{1*}\text{\Huge(}\frac{\text{\large(}c_1(p_2^*\mathcal{P}) +\sum_{i=1}^{n}d_i\Delta_i\text{\large)}^2}{2} - \frac{c_1(\omega_{p_1})\cdot \text{\large(}c_1(p_2^*\mathcal{P}) +\sum_{i=1}^{n}d_i\Delta_i\text{\large)} }{2} + \frac{c_1(\omega_{p_1})^2}{12}  \text{\Huge )}\]
But $\mathcal{P}^{\otimes2} \cong \OO_{\mathcal{C}^1\mathcal{R}_g}$ and we conclude that $2c_1(\mathcal{P}) = 0$. Since the torsion terms disappear in the rational Picard group, we have the equality 
\[ [PD_{\underline{d}}] = -p_{1*}\text{\Large(} \frac{(\sum_{i=1}^{n}d_i\Delta_i)^2}{2} - \frac{(\sum_{i=1}^{n}d_i\Delta_i)\cdot c_1(\omega_{p_1})}{2}+ \frac{c_1(\omega_{p_1})^2}{12}  \text{\Large )}   \] 
hence 
\[  [PD_{\underline{d}}] = \sum_{i=1}^n\frac{d_i^2}{2}\psi_i + \sum_{i=1}^n\frac{d_i}{2}\psi_i - \lambda \]

\end{proof}

We are now ready to compute the class $[\overline{PD}_{\underline{d}}]$ in $\mathrm{Pic}(\cpp_g)$. 

\textbf{Proof of Theorem \ref{prymquad}:} We will denote the Prym-canonical class $[\overline{PD}_{\underline{d}}]$ by 
	\[ [\overline{PD}_{\underline{d}}] = \sum_{j=1}^n \frac{d_j(d_j+1)}{2}\psi_j -\lambda - b_0'\delta_0'-b_0''\delta_0'' - b_0^{\mathrm{ram}}\delta_0^{\mathrm{ram}} - \sum_{\substack{1\leq i \leq g \\  S\subseteq\left\{1,\ldots,n\right\}}} b_{i,S} \delta_{i,S} - \sum_{\substack{1\leq i \leq g-i \\  S\subseteq \left\{1,\ldots,n\right\} }} b_{i,S:g-i} \delta_{i,S:g-i}. \]
Our goal is to compute the coefficients of the boundary divisors. 

We define 
\[ \mathcal{H}^2_g(2\underline{d}, 2^{g-1}) \coloneqq \left\{[C,x_1,\ldots,x_{g+n-1}]\in \cM_{g,g+n-1} \ | \ \OO_C(\sum_{i=1}^{n}2d_i x_i + \sum_{i=n+1}^{g+n-1}2x_i) \cong \omega_C^{\otimes2}\right\} \]
and denote by $\mathcal{Q}_g(2\underline{d}, 2^{g-1})$ the component of $\mathcal{H}^2_g(2\underline{d}, 2^{g-1})$ parametrizing divisors that are not twice the divisor of a holomorphic differential.

We consider the morphism $\mathcal{Q}_g(2\underline{d}, 2^{g-1})\rightarrow \mathcal{C}^n\mathcal{R}_g$ defined as \[[C,x_1,\ldots,x_{g+n-1}] \mapsto [C,x_1,\ldots, x_n, \omega_C\otimes\OO_C(-\sum_{i=1}^{n}d_i x_i -\sum_{i=n+1}^{g+n-1}x_i)]\]
and we immediately observe that the image of this map is the divisor $PD_{\underline{d}}$. Consequently, the closure $\overline{PD}_{\underline{d}}$ in $\cpp_g$ is well understood, see \cite{Daweik-diffcomp}, and we can use the method of \cite[Proposition 1.4]{Mullanek-diff} to compute its class. 

In order to conclude the proof, we first set some notations. For a partition $\underline{m} = (m_1,\ldots,m_n)$ of $g$ with all entries positive, we consider the stratum 
\[ \mathcal{H}_g(\underline{m}, 1^{g-2}) = \left\{[C,x_1,\ldots,x_{g+n-2}]\in \cM_{g,g+n-2} \ | \ \OO_C(\sum_{i=1}^{n}m_i x_i + \sum_{i=n+1}^{g+n-2}x_i) \cong \omega_C\right\} \]
We consider the map $\mathcal{H}_g(\underline{m}, 1^{g-2})\rightarrow \cM_{g,n}$ forgetting the points $x_{n+1}, \ldots, x_{g+n-2}$ and denote $D^g_{\underline{m}}$ its image. We have a similar approach in the case when $\underline{m}$ is a length $n$ partition of $g-1$ with at least one negative entry. We denote 
\[ \mathcal{H}_g(\underline{m}, 1^{g-1}) = \left\{[C,x_1,\ldots,x_{g+n-1}]\in \cM_{g,g+n-1} \ | \ \OO_C(\sum_{i=1}^{n}m_i x_i + \sum_{i=n+1}^{g+n-1}x_i) \cong \omega_C\right\} \]
and we consider the map to $\cM_{g,n}$ forgetting the last $g-1$ entries. We denote its image by $D^g_{\underline{m}}$.

In the notations of Proposition \ref{pullmap1}, Proposition \ref{pullmap2} and Proposition \ref{pullmap3}, we have 
\[ \pi_2^*[\overline{PD}_{\underline{d}}] = PD_{\underline{d'}} + \textup{Boundary divisors}, \] \[\pi_3^*[\overline{PD}_{\underline{d}}] = PD_{\underline{d'}} + \textup{Boundary divisors, and}  \]
 \[ \pi_1^*[\overline{PD}_{\underline{d}}] = 
\begin{cases} D^{n+1-s}_{\underline{d}''} + \textup{Boundary terms}
 & \mathrm{when} \ d_T\geq 2i-2 \\
D^{n+1-s}_{\underline{d}'} + \textup{Boundary terms} & \mathrm{when} \ d_T \leq 2i-4
\end{cases}
\]
where $\underline{d}' = (d_1,\ldots, d_{n-s}, d_T-i)$ and $\underline{d}'' = (d_1,\ldots, d_{n-s}, d_T+1-i)$.
These equalities are satisfied as a consequence of the proof of \cite[Proposition 1.4]{Mullanek-diff}. Because we can consider a multitude of variations of the maps $\pi_1, \pi_2$ and $\pi_3$, and because we computed the coefficients of the $\psi_j$'s in Proposition \ref{smoothdivisors}, we conclude that: 
\[ b_{i,S:g-i} = \frac{(d_S-i)(d_S-i+1)}{2} \ \textup{and}\]
 \[ b_{i,S} = 
\begin{cases} \frac{(d_S-i+1)(d_S-i+2)}{2}
& \mathrm{when} \ d_S\geq i-1 \\
\frac{(d_S-i)(d_S-i+1)}{2} & \mathrm{when} \ d_S \leq i-2
\end{cases}
\]
Moreover, because $\delta_0$ is not one of the boundary terms appearing in the pullback $\pi_1^*$, as remarked in \cite[Proposition 1.4]{Mullanek-diff}, we deduce that $b_0' = 0$. A similar argument for the map $\pi_2^*$ implies further that $b_0'' = b_0' = 0$. The pushforward of the class $[\overline{PD}_{\underline{d}}]$ was computed in \cite[Proposition 1.4]{Mullanek-diff}. We use this to conclude $b_0^{\mathrm{ram}} = \frac{1}{4}$, thus completing the proof.  
\hfill $\square$

\begin{rmk}
	All the Prym-canonical divisorial strata $PD_{\underline{d}}$ are irreducible, see \cite{Lanneau}.
\end{rmk}

We consider the moduli space $\mathcal{Q}_g(\mu)$ parametrizing divisors of quadratic differentials with zero multiplicities given by the partition $\mu$. It is clear that if $g\geq 22$ and $n\geq g$, the stratum $\mathcal{Q}_g(\mu)$ is of general type, because it maps with finite fibers to $\mathcal{M}_{g,l(\mu)-g}$. Mapping to $\mathcal{C}^{l(\mu)-g}\mathcal{R}_g$ instead allows us to find examples of strata of general type in genus as low as $13$. 
\begin{rmk} Let $\mathcal{Q}_g(\mu)$ a stratum with all entries of $\mu$ even and $l(\mu)\geq g$. If $\mathcal{R}_g$ is of general type, then $\mathcal{Q}_g(\mu)$ is also of general type. Similarly, if we allow $\mu$ to have two odd entries and assume $l(\mu)\geq g+2$ we get that $\mathcal{Q}_g(\mu)$ is of general type when $\mathcal{R}_{g,2}$ is.
\end{rmk}

\bibliography{main}
\bibliographystyle{alpha}
\Addresses
\end{document}